\documentclass[12pt]{amsart}

\usepackage{graphics,wrapfig, graphicx}

\newtheorem{lemma}{Lemma}[section]

\newtheorem{corollary}{Corollary}[section]
\newtheorem{main}{Main Theorem}[section]
\newtheorem{proposition}{Proposition}[section]

\newtheorem{remark}{Remark}[section]

\numberwithin{equation}{section} \numberwithin{theorem}{section}
\numberwithin{example}{section} \numberwithin{remark}{section}
\numberwithin{figure}{section} \numberwithin{algorithm}{section}

\newcommand{\ep}{\varepsilon}

\def\bc{\begin{center}}
\def\ec{\end{center}}
\def\ba{\begin{array}}
\def\ea{\end{array}}

\def\bt{\begin{tabular}}
\def\et{\end{tabular}}
\def\be{\begin{equation}}
\def\ee{\end{equation}}

\def\bma{\begin{matrix}}
\def\ema{\end{matrix}}

\def\vu{{\mathbf u}}

\def\RR{{\mathbb R}}

\def\ra{\rightarrow}

\def\na{\nabla}
\def\nc{div\,}
\def\cn{\cdot\nabla}
\def\pa{\partial}

\def\lambdaS{\lambda_{S}}
\def\lambdaM{\lambda_{M}}

\begin{document}

\title[A  sharp blow-up condition for attractive Euler-Poisson equations]
{A sharp local blow-up condition for\\Euler-Poisson equations with attractive forcing}

\author[Bin Cheng]{Bin Cheng}
\address[Bin Cheng]{\newline
        Department of Mathematics\newline
	University of Michigan, 
	Ann Arbor, MI 48109 USA}
\email []{bincheng@umich.edu}
\urladdr{http://www.umich.edu/~bincheng}
\author[Eitan Tadmor]{Eitan Tadmor}
\address[Eitan Tadmor]{\newline
        Department of Mathematics,  Institute for Physical Science and Technology\newline
        and Center of Scientific Computation And Mathematical Modeling (CSCAMM)\newline
        University of Maryland, 
        College Park, MD 20742 USA}
 \email[]{tadmor@cscamm.umd.edu}
\urladdr{http://www.cscamm.umd.edu/\~{}tadmor}

%\urladdr{}

\date{Received by the editors Month day, 200x; accepted for publication (in revised form) Month day, 200x.} %\today


\thanks{\textbf{Acknowledgment.} Research was supported in part by NSF grants 07-07949, 07-57227 and ONR grant N000140910385.}

%%%%%%%%%%%%%%%%%%%%%%%%%%%%%%%%%%%%%%%%%%%%%%%
% Abstract
%%%%%%%%%%%%%%%%%%%%%%%%%%%%%%%%%%%%%%%%%%%%%%%

\begin{abstract}
We improve the recent result of \cite{ChaeTa:star} proving a one-sided threshold condition which leads to finite-time breakdown of the Euler-Poisson equations in arbitrary dimension $n$. 
\end{abstract}
%%%%%%%%%%%%%%%%%%%%%%%%%%%%%%%%%%%%%%%%%%%%%%%%%%%%%%%%%%%%%%%%%%
\maketitle
%\tableofcontents

%%%%%%%%%%%%%%%%%%%%%%
\section{Introduction}
%%%%%%%%%%%%%%%%%%%%%%
The pressure-less Euler-Poisson (EP) equations in dimension $n\ge1$ are
\begin{subequations}
\begin{align}\label{EP10}
\rho_t+\nc(\rho\vu)&=0\\
\label{EP20}
\vu_t+\vu\cn\vu&=k\nabla\Delta^{-1}(\rho-c),
\end{align}
\end{subequations}
governing the unknown density $\rho=\rho(t,x):\RR_+\times \RR^n \mapsto \RR_+$ and velocity 
$\vu=\vu(t,x):\RR_+\times \RR^n \mapsto \RR^n$ subject to initial conditions $\rho(0,x)=\rho_0(x)$ and
$\vu(0,x)=\vu_0(x)$. They involve two constants: a fixed background state $c>0$ such that $\int(\rho-c)dx=0$, and the constant $k$ which parameterizes the repulsive $k>0$ or attractive $k<0$ forcing,  caused by the Poisson potential $\Delta^{-1}(\rho-c)$. The EP system appears in numerous applications including semiconductor, plasma physics ($k>0$) and collapse of interstellar cloud ($k<0$). 

This paper is restricted to the \emph{attractive case}, $k<0$. For simplicity, we set $c=1$, $k=-1$ in (\ref{EP10}), (\ref{EP20}) to arrive at the unit-free EP system,
\begin{subequations}
\begin{align}\label{EP1}
\rho_t+\nc(\rho\vu)&=0,\\
\label{EP2}
\vu_t+\vu\cn\vu&=-\nabla\Delta^{-1}(\rho-1).
\end{align}
\end{subequations}
All discussion in this paper remains valid for EP system with physical parameters $c>0$, $k<0$ upon a simple rescaling argument --- see Corollary \ref{Coro} below.

We are concerned here with the persistence of $C^1$  regularity for solutions of the attractive EP system. Our main theorem reveals a {\it pointwise} criterion on the initial data, a so-called critical threshold criterion \cite{ELT:EPstar,LiTa:spectral,LiTa:restricted}, that leads to finite time blow-up of $\nabla\vu$. It is a sharp, nonlinear quantification of balance between $\nc\vu$ and $\rho$, two competing mechanisms that dictate the $C^1$ regularity of EP flows. Our result also stands out as a generalization of several existing results \cite{ELT:EPstar,ChaeTa:star,LiTa:HYP,LiTa:restricted} for which further discussion is given after the Main Theorem and its corollary.

\begin{main} Consider the $n$-dimensional, attractive Euler-Poisson system (\ref{EP1}), (\ref{EP2}) subject to 
initial data $\rho_0$, $\vu_0$. Then, the solution will lose $C^1$ regularity at a finite time 
$t=t_c < \infty$, if there exists 
a non-vacuum initial state  $\rho_0(\bar{x})>0$ with vanishing initial vorticity,  $\nabla\times\vu_0(\bar{x})=0$, such that the following sup-critical condition is fulfilled,
\begin{subequations}
\begin{equation}\label{eq:sc}
\nc\vu_0(\bar{x}) < sgn(\rho_0(\bar{x})-1)\sqrt{nF(\rho_0(\bar{x}))},
\end{equation}
where
\begin{equation}\label{eq:scb}
F(\rho):=\left\{\begin{array}{ll}
\displaystyle 1+\frac{2\rho}{n-2} - \frac{n\rho^{2/n}}{n-2}, & n\ne2, \\ \\
\displaystyle 1-\rho +\rho\ln\rho, & n=2.
\end{array}\right.
\end{equation}
\end{subequations}
In particular, $\min_x\nc\vu(t,x)\ra-\infty$ and $\max_x\rho(t,x)\ra\infty$ as $t\uparrow t_c$.
\end{main}

\begin{proof} Combine Lemma \ref{comp:lemma} and Lemma \ref{Omega:lemma} below, while noting that the curve
\[
\nc\vu = sgn(\rho-1)\sqrt{nF(\rho)},
\]
is the separatrix along the boundary of the blow-up region $\Omega=\Omega_1\cup \Omega_2$ defined in (\ref{eqs:Omega}) and illustrated in Figure \ref{fig1}.
\end{proof}

We note by passing that, by classical arguments, the force-free Euler system $\vu_t+\vu\cn\vu=0$ exhibits finite time blow-up if and only if there exists at least one {\it negative} eigenvalue of $\nabla\vu_0(\bar{x})$. In the above theorem, however, finite-time blow-up can occur solely depending on the initial profile of $\nc\vu_0$ and $\rho_0$ regardless of individual eigenvalues of $\na\vu_0$.

We also note that, by rescaling $\rho$ to $\rho/c$, $x$ to $\sqrt{-kc}\,x$ and $t$ to $\sqrt{-kc}\,t$, the Main Theorem immediately applies to the EP system  (\ref{EP10}), (\ref{EP20}) with physical parameters. Since the EP system with $k<0$ models the collapse of interstellar cloud, the following corollary reveals a pointwise condition for mass concentration, $\rho\ra\infty$, which interestingly preludes the birth of new stars. 
\begin{corollary}\label{Coro} Consider the Euler-Poisson system (\ref{EP10}), (\ref{EP20}) with $c>0$, $k<0$ subject to 
initial data $\rho_0$, $\vu_0$. Then, the solution will lose $C^1$ regularity at a finite time $t_c < \infty$, if there exists 
a non-vacuum initial state  $\rho_0(\bar{x})>0$ with vanishing initial vorticity,  $\nabla\times\vu_0(\bar{x})=0$, such that the sup-critical condition is fulfilled,
\begin{equation}\label{eq:scc}
\nc\vu_0(\bar{x}) < sgn(\rho_0(\bar{x})-c)\sqrt{-nkcF\left(\frac{\rho_0(\bar{x})}{c}\right)}
\end{equation}
where $F(\cdot)$ is given in (\ref{eq:scb}).
In particular, $\min_x\nc\vu(t,x)\ra-\infty$ and $\max_x\rho(t,x)\ra\infty$ as $t\uparrow t_c$.
%\end{subequations}
\end{corollary}

The concept of Critical Threshold and associated methodology is originated and developed in a series of papers by Engelberg, Liu and Tadmor \cite{ELT:EPstar}, Liu and Tadmor \cite{LiTa:restricted,LiTa:spectral} and more. It first appears in \cite{ELT:EPstar} regarding pointwise criteria for $C^1$ solution regularity of 1D EP system. The key argument in that paper is based on the convective derivative along particle paths $'=\pa_t+\vu\cn$.
It makes possible to obtain  a 2-by-2 ODE system for $u_x$ and $\rho$ along particle paths --- the so-called Lagrangian formulation. Phase plane analysis is then employed to study the finiteness of the ODE solutions and therefore $C^1$ regularity of the PDE solution. Similar results stay valid for Euler-Poisson systems with geometric symmetry in higher dimensions 
\cite{BrWi:Gra,LiTa:spectral}. To treat genuinely multi-D cases, Liu and Tadmor introduce in \cite{LiTa:spectral} the method of spectral dynamics  which relies on the ODE system governing eigenvalues of $$M:=\nabla\vu,$$ which is the velocity gradient matrix, along particle paths. They identify if-and-only-if, pointwise conditions for global existence of $C^1$ solutions to {\it restricted} Euler-Poisson systems. Chae and Tadmor \cite{ChaeTa:star} further extend the Critical Threshold argument to multi-D full Euler-Poisson systems (\ref{EP1}), (\ref{EP2}) with attractive forcing $k<0$. Their result, however, offers a blow-up region $\nabla\times\vu_0=0,\,\nc\vu_0<-\sqrt{-nkc} $ which  is only a subset of the blow-up region  in (\ref{eq:scc}). This subset is  to the left of the solid line $d\leq d^{-}:=-\sqrt{-nkc}$ depicted in figure  \ref{fig1}. Finally, a recent paper by Tadmor and Wei \cite{TaWe:EP} reveals the critical threshold phenomena in 1D Euler-Poisson system with pressure.

When tracking other results on well-posedness of Euler-Poisson equations, we find them commonly relying on (the vast family of) energy methods  and thus fundamentally differ from our pointwise results obtained via the Lagrangian approach. With repulsive force $k>0$, we refer to \cite{Guo:EP, CoGr:EP} for global existence of classical solutions with small data and \cite{Perthame:EP} for nonexistence of global solutions. With attractive force $k<0$, see \cite{Makino:star1} for local regularity of classical solutions and \cite{Makino:star2,MaPe:EPnon} for nonexistence results. Discussion on weak solutions of Euler-Poisson systems can be found in e.g. \cite{Zhang:EP, MaNa:EP, PRV:EP}. We also refer to \cite{Gamba:EP, DeMa:EP,LuSm:whitedwarf, LuSm:star, Rein:2003:star} and references therein for steady-state solutions. Study of Euler-Poisson system with damping relaxation can be found in e.g. \cite{Wang:EP, ChWa:EP, LNX:EP}.

The rest of this paper is organized as following. In Section \ref{sec:spec}, we follow the idea of \cite{ChaeTa:star} to derive along particle paths an ODE system governing the dynamics of eigenvalues for $S:=\displaystyle{1\over2}(M+M^\top)$. This is a variation of the spectral dynamics for $M$ introduced in \cite{LiTa:spectral}. We then derive in Section \ref{sec:comp} a closed $2\times 2$ ODE system (\ref{eqs:dd}) at the cost of turning one equation into inequality. By the comparison principle, this inequality is in favor of blow-up. Thus, with the inequality sign being replaced with equality sign, a modified ODE system is used to yield sub-solutions and to study blow-up scenario for the original system. Section \ref{sec:stab}, devoted to the modified system, reveals the Critical Threshold for such a system. Consequently, a pointwise blow-up condition for the original system is identified.

%%%%%%%%%%%%%%%%%%%%%%%%%%%%%%%%%%%%%%%%%%%
\section{Spectral dynamics}\label{sec:spec}
%%%%%%%%%%%%%%%%%%%%%%%%%%%%%%%%%%%%%%%%%%%

We examine the gradient matrix $M=\nabla\vu$ and its symmetric part, $S={1\over2}\left(\nabla\vu+(\nabla\vu)^\top\right)$.
Both matrices are used to study the spectral dynamics of Euler systems (see e.g. \cite{LiTa:spectral} for $M$ and \cite{ChaeTa:star} for $S$). The relation between the spectra of $M$ and $S$ is described in the following.

\begin{proposition}\label{eig:prop}Let $\{\lambdaM\}$ denote the eigenvalues of $M$ and $\{\lambdaS\}$ for $S$. Then
\begin{align}\label{sum1}
\sum_{\lambdaM}\lambdaM&=\sum_{\lambdaS}\lambdaS=\nc\vu,\\
\label{sum2}\sum_{\lambdaM}\lambdaM^2&=\sum_{\lambdaS}\lambdaS^2-{1\over2}|{\bf \omega}|^2.
\end{align}
Here, $\omega$  is the ${n(n-1)\over2}$ \emph{vorticity} vector which consists of  the off-diagonal entries of $A:={1\over2}\left(\nabla\vu-(\nabla\vu)^\top\right)$.
\end{proposition}

\begin{proof} Use identity $M=S+A$ and the skew-symmetry of $A$,
\[
\sum_{\lambdaM}\lambdaM=tr(M)=tr(S+A)=tr(S)=\sum_{\lambdaS}\lambdaS.
\]
Squaring the last identity we have $M^2=S^2+A^2+AS+SA$ and therefore,
\begin{align*}
\sum_{\lambdaM}\lambdaM^2&=tr(M^2)=tr(S^2+A^2+AS+SA)=\sum_{\lambdaS}\lambdaS^2+tr(A^2).
\end{align*}
Note that $AS+SA$ is skew-symmetric and thus traceless. 
A simple calculation yields $tr(A^2)=-{1\over2}|\omega|^2$. 
\end{proof}

%\begin{remark}\label{eig:remark}
%The advantage of studying $\lambdaS$ is that they are all real numbers due to symmetry of $S$ whereas $\lambdaM$ may take %imaginary values.
%\end{remark}

Following \cite{LiTa:spectral}, we turn to study the dynamics of $M$ along particle paths. Take the gradient of (\ref{EP2}) to find
\begin{equation}\label{eq:M}
M'+M^2\equiv M_t+u\cn M+M^2=-R(\rho-1),
\end{equation}
where $R$ stands for the \emph{Riesz  matrix}, $R=\{R_{ij}\}:=\{\pa_{x_ix_j}\Delta^{-1}\}$. 

The trace of (\ref{eq:M}) then yields that the divergence, $d:=tr(M)$, is governed by
\[
d'=-\sum_{\lambdaM}\lambdaM^2-(\rho-1),
\]
and in view of (\ref{sum2}),
\begin{equation}\label{eq:domega}
d'=-\sum_{\lambdaS}\lambdaS^2 +\frac{1}{2}|\omega|^2 -(\rho-1).
\end{equation}

We now make the first observation regarding the invariance of the vorticity $\omega$:
taking the skew-symmetric part of the $M$-equation (\ref{eq:M}),
\begin{equation}\label{Aeq}
A'+AS+SA=0.
\end{equation}
It follows that if the initial vorticity vanishes, 
$\omega(\bar{x}) \mapsto \nabla \times \vu(\bar{x})=0 $, then by (\ref{Aeq}), $\omega\mapsto \nabla \times \vu$ vanishes along the particle path which emanates from $\bar{x}$. This allows us to decouple the vorticity and divergence dynamics, and (\ref{eq:domega}) implies

\begin{equation}\label{eq:deq}
d'=-\sum_{\lambdaS}\lambdaS^2-(\rho-1), \qquad \nabla\times \vu=0.
\end{equation}
Finally, we use  Cauchy-Schwartz  $\displaystyle \sum\lambdaS^2\le {1\over n}\left(\sum\lambdaS\right)^2={1\over n}d^2$ and the fact that all $\lambdaS$ are real (due to the symmetry of $S$), to deduce the \emph{inequality},
\begin{subequations}\label{eqs:deq}
\begin{equation}\label{dineq}
d'\le-{1\over n}d^2-(\rho-1).
\end{equation}
This, together with the mass equation (\ref{EP1}) which can be written along particle path
\begin{equation}\label{eq:mass}
\rho'=-d\rho,
\end{equation}
\end{subequations}
give us the desired closed system which governs  $(\rho,d)$ along particle paths.

\begin{remark}
The approach pursued in this paper will be based on the \emph{inequality} (\ref{dineq}) and is therefore limited to derivation of a finite time breakdown. To argue the global regularity, one needs to study the underlying \emph{equality} (\ref{eq:deq}), and to this end, to study the trace  $\sum \lambdaS^2$. In the two-dimensional case, for example, one can use
$\sum \lambdaS^2 = d^2/2+\eta^2/2$ to replace (\ref{dineq}) with
\[
d'= - {1\over 2}d^2-{1\over 2}\eta^2-(\rho-1), \qquad \eta:=\lambda_{S,2}-\lambda_{S,1}.
\]
In this framework, global 2D regularity is dictated by  the dynamics  of the \emph{spectral gap}, $\eta=\lambda_{S,2}-\lambda_{S,1}$, which in turn requires the dynamics of the Riesz transform $R(\rho-1)$. 
\end{remark}

%%%%%%%%%%%%%%%%%%%%%%%%%%%%%%%%%%%%%%%%%%%%%%
\section{A comparison principle with a majorant system}\label{sec:comp}
%%%%%%%%%%%%%%%%%%%%%%%%%%%%%%%%%%%%%%%%%%%%%%

The blow-up analysis, driven by the inequalities (\ref{eqs:deq}), 
\begin{subequations}\label{eqs:dd}
\begin{align}\label{eq:dda}
d'&\le-{1\over n}d^2-(\rho-1),\\ \label{eq:ddb}\rho'&=-d\rho.
\end{align}
\end{subequations}
is carried out by standard comparison with the majorant system
\begin{subequations}\label{eqs:ee}
\begin{align}\label{eq:eea}
e'&=-{1\over n}e^2-(\zeta-1),\\ \label{eq:eeb}
\zeta'&=-e\zeta.
\end{align}
\end{subequations}

The following proposition guarantees the monotonicity of the solution operator associated with (\ref{eqs:dd}).

\begin{lemma}\label{comp:lemma}
The following monotone relation between system (\ref{eqs:dd}) and system (\ref{eqs:ee}) is invariant forward in time,
\begin{equation}\label{comp1}
\left\{\begin{array}{l}
d(0) < e(0) \\
0<\zeta(0)<\rho(0)
\end{array}\right. \mbox{ implies }\ 
\left\{
\begin{array}{l}
d(t) < e(t)\\
0<\zeta(t)<\rho(t)
\end{array}\right.
\mbox{ for }t\ge 0,
\end{equation}
as long as all solutions remain finite on time interval $[0,t]$.
\end{lemma}

\begin{proof} Invariance of positivity of $\zeta$ is a direct consequence of (\ref{eq:eeb}) and finiteness of $e$. The rest can be proved by contradiction. Suppose $t_1$ is the earliest time when (\ref{comp1}) is violated. Then, 
\begin{equation}\label{Rr}
\zeta(t_1)=\zeta(0)\exp\left(-\int_0^{t_1}e(t)dt\right)<\rho(0)\exp\left(-\int_0^{t_1}d(t)dt\right)=\rho(t_1).
\end{equation}
Therefore, we are left with only one possibility $e(t_1)=d(t_1)$. However,  subtracting (\ref{eq:dda}) from (\ref{eq:eea}),
\begin{equation}\label{eq:abc}
(e-d)'\ge-{1\over n}(e^2-d^2)-(\zeta-\rho).
\end{equation}
Setting $t=t_1$ in the above inequality, we find that 
\[
LHS \mbox{ of } (\ref{eq:abc}) = \big(e(t_1)-d(t_1)\big)' \le 0,
\]
 since $e(t)-d(t)>0$ for all $t<t_1$;
but this contradicts the positivity of the expression on the right of (\ref{eq:abc}), for by (\ref{Rr})
\[
RHS \mbox{ of } (\ref{eq:abc}) =0-\left[\zeta(t_1)-\rho(t_1)\right]>0.
\]
\end{proof}

In the next section, we employ phase plane analysis on the modified system (\ref{eqs:ee}). When translated in terms of the original system (\ref{eqs:dd}), however, such analysis can only yield blow-up results and is insufficient for global existence results. In other words, estimate (\ref{comp1}) is only useful for proving $d\searrow-\infty$, the key mechanism for blow-up of $C^1$ solutions.

%%%%%%%%%%%%%%%%%%%%%%%%%%%%%%%%%%%%%%%%%%%%%%%%%%%%%%%%%%%%%%%%%%
\section{Stability analysis of the majorant system}\label{sec:stab}
%%%%%%%%%%%%%%%%%%%%%%%%%%%%%%%%%%%%%%%%%%%%%%%%%%%%%%%%%%%%%%%%%%%
We shall prove the blow-up of the majorant system (\ref{eqs:ee}), $e(t) \rightarrow -\infty$ as $t\uparrow t_c$, which in turn,
by lemma \ref{comp:lemma} implies $d(t) \rightarrow -\infty$. Abusing notations, we express the majorant system in terms of the original variables $(e,\zeta) \mapsto (d,\rho)$: 
\begin{subequations}\label{eqs:dm}
\begin{align}\label{dm}
d'&=-{1\over n}d^2-(\rho-1),\\ 
\label{rm}\rho'&=-d\rho.
\end{align}
\end{subequations}

The (in-)stability analysis of (\ref{eqs:dm}) hinges on the path invariants of this system. To this end, we use the same
$q$-transformation  employed in \cite{LiTa:rotation,LiTa:restricted}: setting $q:=d^2$ and differentiate along the path
$\{(t,X(a,t))\ | \ X_t(a,t)=u(t,X(a,t)), X(a,0)=a\}$, we find
\[
\frac{dq}{d\rho}=2d\frac{d'}{\rho'}= \frac{2}{n\rho}q+2\left(1-\frac{1}{\rho}\right),
\]
which yields
\[
{d\over d\rho}\left(q\rho^{-{2\over n}}\right)=2(1-\rho^{-1})\rho^{-{2\over n}}.
\]
Upon integration, we arrive at the  following key observation.
\begin{lemma}
The majorant system (\ref{eqs:dm}) is equipped with the path invariant, 
\[
I(d(t),\rho(t))=I(d_0,\rho_0),
\]
along each path $(t,x(t))$ initiated with a non-vacuum state $(d_0, \rho_0>0)$. Here,
\begin{equation}\label{II}
I(d,\rho):=d^2\rho^{-{2\over n}}-2\int_1^\rho(1-r^{-1})r^{-{2\over n}}\,dr=\rho^{-{2\over n}}\left(d^2-nF(\rho)\right),
\end{equation}
where $F(\cdot)$ is specified in (\ref{eq:scb}).
\end{lemma}

It is simple calculation to show that the majorant system (\ref{eqs:dm}) admits three distinct critical points (see figure \ref{fig1}):
\[
(d^*,\rho^*):=(0,1),\quad(d^\pm,\rho^{\pm}):=(\pm\sqrt{n},0).
\]
and that $(0,1)$ is a saddle point, $(-\sqrt{n},0)$ a nodal source and $(\sqrt{n},0)$ a nodal sink. The separatrix is given by the zero level set $I(d,\rho)=0$. Moreover, the right branch of the separatrix, $d=\sqrt{nF(\rho)}$ connects critical points $(0,1)$ and $(\sqrt{n},0)$ while the left branch,  $d=-\sqrt{nF(\rho)}$ connects $(0,1)$ and $(-\sqrt{n},0)$. 

By inspection of the phase plane in figure \ref{fig1}, we postulate the following invariant region of finite-time blow-up for the modified system (\ref{eqs:dm}),
\begin{subequations}\label{eqs:Omega}
\begin{equation}\label{Omega}
\Omega=\Omega_1\cup\Omega_2=\{(d,\rho)\;|\;d<sgn(\rho-1)\sqrt{nF(\rho)}\}
\end{equation}
where
\begin{eqnarray}\label{eq:Omega1}
\Omega_1&:=&\{(d,\rho)\;|\;I(d,\rho)>0 \mbox{ and }d<0\mbox{ and }\rho>0\},\\
\label{eq:Omega2}
\Omega_2&:=&\{(d,\rho)\;|\;I(d,\rho)<0 \mbox{ and }\rho>1\}.
\end{eqnarray}

\end{subequations}

\begin{figure}[htbp]
\begin{center}\includegraphics[height=200pt,width=400pt]{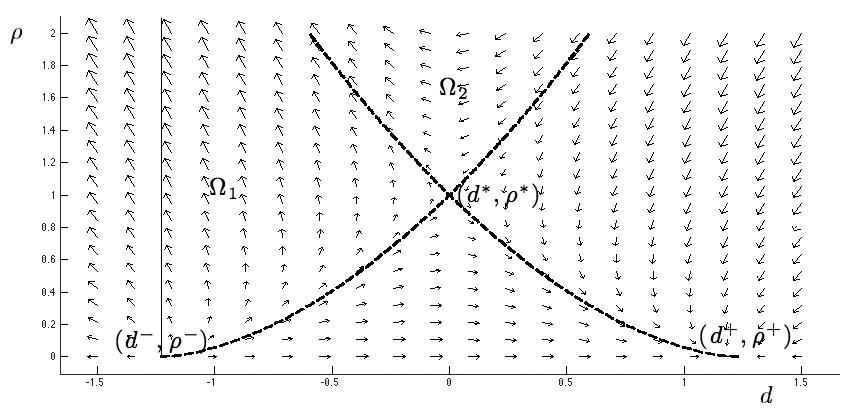}\end{center}
\caption{Phase plane of (\ref{eqs:dm}) with blow-up region $\Omega_1\cup\Omega_2$ which extends the Chae-Tadmor region \cite{ChaeTa:star} $d\leq d^{-}$.}
\label{fig1}
\end{figure}

\begin{lemma}\label{Omega:lemma}
Consider the modified system (\ref{eqs:dm}), equipped with initial data $(d_0,\rho_0)$. If $(d_0,\rho_0)\in\Omega$, then $\nc\vu\rightarrow-\infty$ and $\rho\to\infty$ at a finite time.
\end{lemma}

\begin{proof} 
We begin by recalling (\ref{eq:scb}), consult (\ref{II}),
\[
F(\rho)=\frac{2}{n}\rho^{2\over n}\int_1^{\rho}(1-r^{-1})r^{-{2\over n}}\,dr.
\]
Clearly, $F(1)=F'(1)=0$ and a simple calculation shows that 
$F''(\rho)=\dfrac{2}{n}\rho^{\frac{2}{n}-2}$, which implies that $F(\rho)$ is a strictly convex function of positive $\rho$ and attains its only minimum at $\rho=1$, 
\begin{equation}\label{est:1}
F(\rho) \ge F(1)=0.
\end{equation}

We shall also utilize the invariance of (\ref{II}) 
\begin{equation}\label{dFrho}
d^2-nF(\rho)=\rho^{2\over n}I_0, \qquad I_0=I(d_0,\rho_0).
\end{equation}

We now turn to discuss the two possible blow-up scenarios, depending whether the initial data $(d_0,\rho_0)$ belong to the blow-up regions  $\Omega_1$ or $\Omega_2$  given in (\ref{eqs:Omega}).
\medskip

%\begin{itemize}
%\item[
\noindent
{\it Case \#1}. Assume that $(d_0,\rho_0)\in\Omega_1$ so that the invariant $I$ remains a \emph{positive} constant\[I>0.\] In this case, $d$ remains negative, for otherwise, setting $d=0$ in (\ref{dFrho}) would result in $F(\rho)=-\rho^{2\over n}I/n<0$, violating (\ref{est:1}). Thus, (\ref{dFrho}) and (\ref{est:1}) yield an upper bound,
\begin{equation*}\label{d:bound}
d\le-\rho^{1\over n}\sqrt{I}.
\end{equation*}
Then, by (\ref{rm}), we have a Riccati type of equation $\rho'\ge\sqrt{I}\rho^{1+{1\over n}}$ for which the solution exhibits blow-up $\rho\rightarrow+\infty$ and the divergence $d=\nc\vu$  approaches $-\infty$ at a  finite time due to (\ref{dFrho}).
\medskip

%\item[
\noindent
{\it Case \#2}. Assume that $(d_0,\rho_0)\in\Omega_2$ so that the invariant $I$ remains  a \emph{negative} constant\[I<0.\] In this case,  $\rho-1$ remains positive, for otherwise setting $\rho=1$ in (\ref{dFrho}) would result in $F(1)=(d^2-I)/n>0$ in contradiction to (\ref{est:1}).
Now, for $\rho>1$ we have
\[
F(\rho)=\frac{2}{n}\rho^{2/n}\int_1^\rho \left(1-\frac{1}{r}\right)\frac{1}{r^{2/n}}dr
\leq \frac{2}{n}\rho^{2/n}(\rho-1).
\] 
This together with (\ref{dFrho})  yield 
\[
\frac{2}{n}\rho^{2/n}(\rho-1)\ge F(\rho)=\frac{1}{n}\left(d^2-\rho^{2/n}I\right)\ge -\frac{1}{n}\rho^{2/n}I
\]
and the lower bound, $\rho-1\ge -I/2$ follows.
Thus, by (\ref{dm}), we end up with a Riccati type of equation 
\[
d'\le-\dfrac{d^2}{n}+\dfrac{I}{2}.
\]
Since the invariant $I$ remains a negative constant, the solution exhibits blow-up $d=\nc\vu\rightarrow-\infty$ at a   finite time even if initially $d_0>0$. The density $\rho$ also approaches $\infty$ in finite time due to (\ref{dFrho}).
%\end{itemize}
\end{proof}

The last step of proving the Main Theorem is just to combine the comparison principle in Lemma \ref{comp:lemma} with the above lemma. We notice that $\Omega$ is an open set and thus given any initial data $(d_0,\rho_0)\in\Omega$ for the original system, we can always find $\ep>0$ and initial data $(d_0+\ep,\rho_0-\ep)\in\Omega$ for the modified system. This latter initial data will lead to finite time blow-up of the modified system and therefore, by lemma \ref{comp:lemma}, initial data $(d_0,\rho_0)\in\Omega$ will lead to finite time blow-up of the original system.

\bibliographystyle{amsplain}
\bibliography{Cheng-Tadmor_EulerPoisson}

\end{document}